\documentclass[a4paper]{article}
\usepackage{amsmath, amsfonts}
\usepackage{amssymb, latexsym}
\usepackage{amsthm}
\usepackage{mathtools}
\usepackage{cite}

\usepackage[shortlabels]{enumitem}
\setlist[enumerate]{leftmargin=*,align=left,labelindent=\parindent}

\usepackage{tikz-cd}
\tikzset{close/.style={near start,outer sep=-10pt}}
\usepackage[%dvipdfmx,
colorlinks=true,       % false: boxed links; true: colored links
linkcolor=blue,          % color of internal links (change box color with linkbordercolor)
citecolor=blue,        % color of links to bibliography
plainpages=false,      % do page number anchors as plain Arabic
pdfpagelabels,
]{hyperref}

\newcommand{\defeql}{\overset{\mathrm{def}}{\,=\,}}
\newcommand{\defequiv}{\overset{\mathrm{def}}{\iff}}

\newcommand{\FPA}{\mathrm{FPA}}
\newcommand{\SGA}{\mathrm{SGA}}
\newcommand{\NID}{\mathrm{NID}}
\newcommand{\NIDbi}{\mathrm{NID}_\mathrm{bi}}
\newcommand{\NIDn}[1]{\mathrm{\NID}_{#1}}
\newcommand{\NIDf}{\mathrm{\NID}_{< \omega}}

\newcommand{\Nat}{\omega}

\newcommand{\mv}{\mathrm{mv}}
\newcommand{\pow}{\mathrm{Pow}}
\newcommand{\fin}{\mathrm{Fin}}
\newcommand{\ext}{\mathop{\mathrm{ext}}\nolimits}
\newcommand{\id}{\Delta}
\newcommand{\meets}{\between}
\newcommand{\op}{\mathrm{op}}
\newcommand{\cov}{\vartriangleleft}
\newcommand{\amp}{\mathrel{\&}}
\newcommand{\Pt}{\mathcal{P}\mathit{t}}

\newcommand{\One}{\mathbf{1}}

\newcommand{\BP}{\mathbf{BP}}
\newcommand{\Rel}{\mathbf{Rel}}
\newcommand{\CSpa}{\mathbf{CSpa}}

\newcommand{\CZF}{\mathbf{CZF}}
\newcommand{\ECST}{\mathbf{ECST}}

\newcommand{\Model}[1]{\mathcal{M}(#1)}
\newcommand{\Conv}{\mathrm{Conv}}
\DeclareMathOperator*{\medwedge}{{\textstyle{\bigwedge}}}

\newtheorem{theorem}{Theorem}%[section]
\newtheorem{proposition}[theorem]{Proposition}
\newtheorem{lemma}[theorem]{Lemma}

\theoremstyle{definition}
\newtheorem{definition}[theorem]{Definition}
\theoremstyle{remark}
\newtheorem{remark}[theorem]{Remark}

\numberwithin{equation}{section}

\title{Equivalents of the finitary non-deterministic inductive definitions}
\usepackage{authblk}

\author{Ayana Hirata\hspace{.8em}
  Hajime Ishihara\hspace{.8em}
  Tatsuji Kawai\hspace{.8em}
  Takako Nemoto \\
\normalsize
School of Information Science\authorcr
Japan Advanced Institute of Science and Technology \authorcr
1-1 Asahidai, Nomi, Ishikawa 923-1292, Japan
}
\date{}
\begin{document}
\maketitle

\begin{abstract}
We present statements equivalent to some fragments of the
principle of
non-deterministic inductive definitions ($\NID$) by van den
Berg~(2013), working in a weak subsystem of constructive set theory
$\CZF$.
We show that several statements in constructive topology which were
initially proved using $\NID$  are equivalent to 
the elementary and finitary $\NID$s.
We also show that the finitary $\NID$ is equivalent to its binary
fragment and that the elementary $\NID$ is equivalent to a
variant of $\NID$ based on the notion of biclosed subset.
Our result suggests that proving these statements in
constructive topology requires genuine extensions of $\CZF$ with
the elementary or finitary~$\NID$.

%%%%%%%%%%%%%%%%%%%%%%%%%%%%%%%%%%%%%%%%%%%%%%%%%%%%
%%%%%%%%%%%%%%%%%%%%%%%%%%%%%%%%%%%%%%%%%%%%%%%%%%%%
%
% 03E70 Nonclassical and second-order set theories 
% 03F50 Metamathematics of constructive systems
% 54A05 Topological spaces and generalizations (closure spaces, etc.)
% 06D22  Frames, locales
%
\end{abstract}

%%%%%%%%%%%%%%%%%%%%%%%%%%%%%%%%%%%
\section{Introduction}
%%%%%%%%%%%%%%%%%%%%%%%%%%%%%%%%%%%
Many objects in mathematics are defined as subsets of some given
set, e.g., an open set of a topological space; a prime ideal
of a commutative ring; a sub-algebra of a certain algebraic structure.  The
totality of these objects, however, does not necessarily form a set in
predicative constructive foundations such as Martin-L\"of's type theory
\cite{MartinLofMLTT} or  Aczel's constructive set theory $\CZF$
\cite{Aczel-Rathjen-Note}. In particular, the lack of power-sets in
these foundations makes some of the standard constructions in general topology substantially
difficult to carry out, which requires a certain amount of ingenuity
\cite{AspectofTopinCZF,IshiharaPalmgrenQuotient,PalmgrenCoequalisers,PalmgrenMaxPartialPt}.

The crucial element of these predicative results consists in
constructing a subset of the totality of a certain type of objects,
called a \emph{generating subset}, in such a way that every object of that type can be
expressed as the union of the elements of the generating subset.
The problems of constructing such generating subsets
in constructive topology motivated 
van den Berg \cite{vandenBergNID} and
Aczel, Ishihara, Nemoto, and
Sangu~\cite{SGA} to
independently introduce principles of $\CZF$ which allow us to show
that a wide range of collections of mathematical objects are
set-generated.

The focus of this paper is on the principle introduced by van den Berg
\cite{vandenBergNID}, called \emph{non-deterministic inductive
definitions} ($\NID$), or more specifically, its elementary and
finitary fragments. The $\NID$ principle asserts that the class of
models of an infinitary propositional theory consisting of
formulas (or \emph{rules}) of the form $\bigwedge U \rightarrow \bigvee V$
has a generating subset, where $U$ and $V$ are subsets of the set of
propositional variables. The elementary and finitary $\NID$ principles are
obtained by restricting the propositional theory to those rules
whose premise is singleton and
finitely enumerable, respectively. In fact, van den Berg
\cite{vandenBergNID} showed that the finitary $\NID$ is equivalent to
the principle introduced by Aczel et al.\ \cite{SGA}, called the
\emph{set generation axiom} ($\SGA$).

The purpose of this paper is to extend the scope of reverse
mathematics, classical \cite{SimpsonSubsystem2ndOrderLogic} or constructive
\cite{ConstRevMatheCompactness,VeldmanFANKleene},
in a set-theoretic foundation. Here, we develop the
reverse mathematics of the $\NID$ principle initiated in
\cite{IshiharaNemotoNIDFullness} much further by showing that several statements in
constructive topology which were initially proved using $\NID$ or
$\SGA$ are in fact equivalent to the elementary $\NID$ or the finitary
$\NID$.
Specifically, we show that the elementary $\NID$ is equivalent to 
(1)~completeness and cocompleteness of the category of basic pairs by
Sambin~\cite{Sambin_BP_book} and (2) the existence of weak equalisers
in the category of sets and relations. Moreover,
the finitary $\NID$ is equivalent to
(1)~completeness and cocompleteness of the category of concrete
spaces and (2)~set-generation of the class of formal points of an
inductively generated formal topology.
We also show that 
the finitary $\NID$ is equivalent to its binary fragment and that the
elementary $\NID$ is equivalent to a symmetric variant of $\NID$,
called $\NIDbi$, formulated with respect to the notion of biclosed subset.

Our result suggests that proving these 
statements in constructive topology requires genuine extensions of $\CZF$ with  the elementary
$\NID$ or the finitary $\NID$, which are thought to be independent of $\CZF$
(cf.\ van den Berg~\cite[Section~8]{vandenBergNID}).
It remains to settle the exact relation between $\CZF$, and its
extensions with the elementary
$\NID$ or the finitary $\NID$, which would also settle the relative
strength of the equivalents of these principles.

\paragraph{Organisation} Section~\ref{sec:ECST} introduces the
base set theory for our work; Section~\ref{sec:NID} recalls the
$\NID$ principle and its relation to $\SGA$; Section~\ref{sec:NID1}
gives equivalents of the elementary $\NID$; and Section~\ref{sec:NIDf}
gives equivalents of the finitary $\NID$.

%%%%%%%%%%%%%%%%%%%%%%%%%%%%%%%%%%%
\section{Elementary constructive set theory} \label{sec:ECST}
%%%%%%%%%%%%%%%%%%%%%%%%%%%%%%%%%%%
We work in a weak subsystem of $\CZF$, called the \emph{elementary
constructive set theory} $\ECST$  \cite{CTSbook}, where none of the
known fragments of the $\NID$ principle seems to be derivable.

The language of $\ECST$ contains variables for sets and binary
predicates $=$ and $\in$. The axioms and rules of $\ECST$ are the
axioms and rules of intuitionistic predicate logic with equality, and
the following set-theoretic axioms:
\begin{description}
    \item[Extensionality:]
    $
     \forall a \forall b \left( \forall x \left( x \in a \leftrightarrow
    x \in b\right) \rightarrow a = b \right).
    $
  \item[Pairing:]
    $
    \forall a \forall b\exists y \forall u \left( 
    u \in y \leftrightarrow u = a \vee u = b \right).
    $
  \item[Union:]
    $
    \forall a \exists y \forall x  \left( x \in y \leftrightarrow
    \exists u \in a \left( x \in u \right) \right).
    $
  \item[Restricted Separation:]
     $
     \forall a \exists b   
    \forall x \left( x \in b \leftrightarrow x \in
    a \wedge
    \varphi(x)\right)
    $ \\[.5em]
    where $\varphi(x)$ is restricted.
    Here, a formula is said to be \emph{restricted} if all quantifiers
    in the formula occur in the forms $\forall x \in a$
    or $\exists x \in a$.

  \item[Replacement:]
      $
      \forall a \left( \forall x \in a \exists ! y\, \varphi(x,y)
      \rightarrow \exists b \forall y \left( y \in b \leftrightarrow
      \exists x \in a \, \varphi(x,y) \right) \right)
      $\\[.5em]
    where $\varphi(x,y)$ is any formula.
  \item[Strong Infinity:]
    \raisebox{-0.6\baselineskip}{$
    \begin{multlined}
      \exists a \left[ 
        0 \in a \wedge \forall x \left( x \in a \rightarrow x + 1 \in
        a\right) \right.\\
        \qquad\qquad \left.
        \wedge \forall y \left( 0 \in y \wedge \forall x \left( x \in
        y \rightarrow x + 1 \in y
        \right) \rightarrow a \subseteq y \right)
      \right]
    \end{multlined}
  $}
    \\[.5em]
      where $x + 1$ denotes $x \cup \left\{ x \right\}$ and $0$ is the
      empty set $\emptyset$. The set $a$ asserted to exist
      will be denoted by $\Nat$.
\end{description}
This completes the description of $\ECST$.

The constructive Zermelo--Fraenkel set theory $\CZF$
\cite{Aczel-Rathjen-Note} is obtained from $\ECST$ by substituting
Strong Collection for Replacement and adding Subset Collection and
$\in$-Induction. For the details of these axioms, the reader is
referred to Aczel and Rathjen \cite{Aczel-Rathjen-Note,CTSbook}.

In $\ECST$, Subset Collection implies
\begin{description}
\item[Fullness:]
  $
   \forall a \forall b \exists c 
   \left[ c \subseteq \mv(a,b)  \wedge \forall s \in
  \mv(a,b)  \exists r \in c \left( r \subseteq s\right)\right],
  $
\end{description}
where $\mv(a,b)$ is the class of total relations from $a$ to $b$.
In practice, Fullness is very important as it implies
Exponentiation, which asserts that the class $B^{A}$ of functions between
sets $A$ and $B$ is a set.
In particular, Fullness implies the following weak form of
Exponentiation:
\begin{description}
    \item[Finite Powers Axiom ($\FPA$):] For any set $S$, the class $S^{n}$ of functions from 
  $\left\{ 0,\dots,n-1 \right\}$ to $S$ is a set for all $n \in \omega$.
\end{description}
The extension of $\ECST$ with $\FPA$ is denoted by $\ECST + \FPA$.

A notable consequence of $\FPA$ is that the class $\fin(S)$ of finitely enumerable subsets
of a set $S$ is a set. Here,
a set $A$ is \emph{finitely enumerable} if there is a 
surjection $f \colon \left\{ 0,\dots,n-1 \right\} \to A$ for some
$n \in \Nat$. Note that we can decide whether a finitary enumerable set
is empty or inhabited by inspecting the domain of $f$.

%%%%%%%%%%%%%%%%%%%%%%%%%%%%%%%%%%%
\section{$\NID$ principles} \label{sec:NID}
%%%%%%%%%%%%%%%%%%%%%%%%%%%%%%%%%%%
The subjects of our investigation are classes closed under some sets of
rules on a set.
$\NID$ principles say that such a class has a generating subset
$G$ so that every member of the class arises as the union of elements of
$G$.
\begin{definition}
A \emph{rule} on a set $S$ is a pair $(a,b)$ of subsets of $S$.
A rule $(a,b)$ is said to be
  \emph{nullary} if $a$ is empty,
  \emph{elementary} if $a$ is  singleton, and
  \emph{finitary} if $a$ is  finitely enumerable.
A subset $\alpha \subseteq S$ is said to be \emph{closed under} a rule $(a,b)$ if
\begin{equation*}
  a \subseteq \alpha \implies b \meets \alpha,
\end{equation*}
where $b \meets \alpha$ means that $b \cap \alpha$ is inhabited.
If $R$ is a set of rules on $S$, then a subset $\alpha \subseteq S$ is said to
be \emph{$R$-closed} %(respectively, \emph{$R$-biclosed}) 
if it is closed %(respectively, biclosed)
under every rule in $R$.
\end{definition}

\begin{definition}
  Let $S$ be a set and $\pow(S)$ be the class of subsets of
  $S$. A subclass $\mathcal{C}$ of $\pow(S)$ is said to be
\emph{set-generated} if there exists a subset $G \subseteq \mathcal{C}$, called a
\emph{generating subset}, such that
\[
  \forall \alpha \in \mathcal{C} \forall x \in \alpha \exists \beta \in G 
  \left( x
\in \beta \subseteq \alpha \right).
\]
The principle $\NID$ %(respectively, $\NIDbi$)
reads:
\begin{description}
  \item[$\NID$:]
    For each set $S$ and a set $R$ of rules on $S$, the class of
    $R$-closed %(respectively, $R$-biclosed)
    subsets of $S$ is set-generated.
\end{description}
The \emph{nullary}, \emph{elementary}, and \emph{finitary} $\NID$
are the principles obtained from $\NID$ by restricting $R$ to nullary,
elementary, and finitary rules, respectively. These principles are denoted by
$\NIDn{0}$, $\NIDn{1}$, and $\NIDf$.
\end{definition}

\begin{remark}
  \label{rem:NID0Fullness}
$\NIDf$ clearly implies $\NIDn{1}$. Moreover,
Ishihara and Nemoto~\cite{IshiharaNemotoNIDFullness} showed that $\NIDn{1}$ implies
$\NIDn{0}$ and that $\NIDn{0}$
is equivalent to Fullness over $\ECST$. In particular, $\NIDn{0}$ implies $\FPA$.
\end{remark}

We recall the connection between $\NIDf$
and the \emph{set generation axiom} ($\SGA$) introduced by 
Aczel et al.\ \cite{SGA}.
\begin{definition}
  For any set $S$, a subclass $\mathcal{C}$ of $\pow(S)$ is said to be
  \emph{strongly set-generated} if there exists a subset $G \subseteq \mathcal{C}$ such that
  \[
    \forall \alpha \in \mathcal{C} \forall \sigma \in
  \fin\left(\alpha \right) \exists \beta \in G
  \left( \sigma \subseteq \beta \subseteq \alpha  \right).
  \]
  The principle $\SGA$ reads:
  \begin{description}
    \item[$\SGA$:]
      For each set $S$ and each subset $Z \subseteq \fin(S) \times
      \pow(\pow(S))$, the class
      \[
      \mathcal{M}(Z) = \left\{ \alpha \in \pow(S) \mid \forall (\sigma,
      \Gamma) \in Z \left( \sigma \subseteq \alpha \rightarrow
      \exists U \in \Gamma \left( U \subseteq \alpha \right) \right) \right\}
      \]
      of models of $Z$ is strongly set-generated.
  \end{description}
A subset $Z \subseteq \fin(S) \times
      \pow(\pow(S))$ is called a \emph{generalised geometric theory} 
over $S$ (of rank $1$), and its element is written $\medwedge \sigma \vdash
\bigvee_{U \in \Gamma}\bigwedge U$ instead of $(\sigma,\Gamma)$.
\end{definition}
The principle $\SGA$ looks stronger than $\NIDf$, but they are
equivalent.
\begin{proposition}[{van den Berg \cite[Theorem 7.3]{vandenBergNID}}]
  $\SGA$ and $\NIDf$ are equivalent over $\ECST$.
  \qed
\end{proposition}
\begin{remark}\label{rem:Rankn}
In $\SGA$, we can take $Z$ to be a set of
elements of the form
\[
  \medwedge \sigma \vdash
  \bigvee_{U_{n-1} \in \Gamma_{n-1}} \bigwedge_{\Gamma_{n-2} \in
  U_{n-1}}
  \cdots \bigvee_{U_{0} \in \Gamma_{0}}\bigwedge U_{0}
\]
where $\sigma \in \fin(S), U_{0} \in \pow(S), U_{1} \in
\pow(\pow(\pow(S))),\dots$, and the right-hand-side
is a finite nesting of $\bigvee \bigwedge$ pairs.
The resulting principle is still equivalent to 
$\NIDf$ over $\ECST$.
See van den Berg \cite[Theorem 4.2]{vandenBergNID}.
\end{remark}

%%%%%%%%%%%%%%%%%%%%%%%%%%%%%%%%%%%%%%%%%%%%
\section{Elementary $\NID$}\label{sec:NID1}
%%%%%%%%%%%%%%%%%%%%%%%%%%%%%%%%%%%%%%%%%%%%
In this section, we give some statements equivalent to 
$\NIDn{1}$.

\subsection{$\NIDbi$ principle}
\label{sec:NIDbi}
We introduce a symmetric variant of $\NID$, which 
 seems to be quite natural and useful in practice (cf.\ Section
\ref{sec:WeakEqlRel} and Section \ref{sec:EqlBP} below). 

\begin{definition}
  Let $(a, b)$ be a rule on a set $S$. A subset $\alpha
  \subseteq S$ is said to be \emph{biclosed} under $(a,b)$ if
  \[
    a \meets \alpha \iff b \meets \alpha.
  \]
  If $R$ is a set of rules on $S$, then a subset $\alpha \subseteq S$ 
  is said to be \emph{$R$-biclosed} if $\alpha$ is biclosed under
  every rule in
  $R$.
  Then, the principle $\NIDbi$ reads:
  \begin{description}
    \item[$\NIDbi$:]
    For each set $S$ and a set $R$ of rules on $S$, the class
    of $R$-biclosed subsets of S is set-generated.
  \end{description}
\end{definition}

\begin{proposition}
  $\NIDn{1}$ is equivalent to $\NIDbi$.
\end{proposition}
\begin{proof}
  Assume $\NIDn{1}$, and let $R$ be a set of rules on a set $S$. Define a set $R'$ of
elementary rules on $S$ by
\[ 
  R' \defeql \left\{   (\{x\},b) \mid (a,b) \in R, x \in a \} \cup \{ (\{y\},a)
\mid (a,b) \in R, y \in b \right\}.
\]
Then, a subset $\alpha \subseteq S$ is $R$-biclosed if and only if it is $R'$-closed.
This proves one direction.

Conversely, assume $\NIDbi$, and let $R$ be a set of elementary rules on a set $S$.
Define another set $R'$ of rules on $S$ by
\[
  R' \defeql \left\{ (a \cup b, b) \mid (a,b) \in R\right\}.
\]
Then, a subset $\alpha \subseteq S$ is $R$-closed if and only if it is
$R'$-biclosed. %Thus $\NIDbi$ implies $\NIDn{1}$.
\end{proof}

\subsection{Weak equalisers in the category of sets and relations}
\label{sec:WeakEqlRel}
We show that $\NIDn{1}$ is equivalent to the
existence of weak equalisers in the category of sets and relations. 

Let $\Rel$ be the category of sets and relations between them:
the identity on a set $X$ is a diagonal relation 
\[
  \id_{X} \defeql \left\{ (x,x) \mid x \in X \right\},
\]
and the composition of morphisms is the relational composition.

\begin{definition}
  Let $f, g \colon A \to B$ be a parallel pair of morphisms in a category
  $\mathbb{C}$.
  A \emph{weak equaliser} of $f$ and $g$ is an object $E$ together
  with a morphism $e \colon E \to A$ such that
  \begin{enumerate}
    \item $f \circ e = g \circ e$,
    \item for any morphism $h \colon C \to A$ such that $f \circ h = g
      \circ h$, there exists a morphism $\overline{h} \colon C \to E$
      such that $e \circ \overline{h} = h$.
  \end{enumerate}
\end{definition}

\begin{proposition}
  \label{prop:NIDbiRelWEq}
  The following are equivalent over $\ECST$:
  \begin{enumerate}
    \item\label{prop:NIDbiRelWEq1} $\NIDbi$;
    \item\label{prop:NIDbiRelWEq2} $\Rel$ has weak equalisers.
  \end{enumerate}
\end{proposition}
\begin{proof}
  (\ref{prop:NIDbiRelWEq1} $\rightarrow$ \ref{prop:NIDbiRelWEq2}) Assume $\NIDbi$, and let $r_{1}, r_{2} \subseteq X
  \times Y$
  be a parallel pair of relations. Consider a class
  \[
    \mathcal{E} \defeql \left\{ U \in \pow(X) \mid r_{1}U =
    r_{2}U \right\},
  \]
  where
    $
    r_{i}U \defeql \left\{ y \in Y \mid  \exists x \in U \left(
      x \mathrel{r_{i}} y \right)\right\} \; (i = 0,1)
    $.
  Define a set of rules on $X$ by
  \[
    R \defeql \left\{ (r_{1}^{-}y, r_{2}^{-}y ) \mid  y \in Y
  \right\},
  \]
  where
    $
    r_{i}^{-}y \defeql \left\{ x \in X \mid 
      x \mathrel{r_{i}} y \right\} \; (i = 0,1)
    $.
    Since
  \begin{align*}
    U \in \mathcal{E} 
    \iff
     \forall y \in Y  \left( y \in r_{1}U \leftrightarrow y \in r_{2}U \right) 
    \iff
     \forall y \in Y \left(r_{1}^{-}y  \meets U
    \leftrightarrow  r_{2}^{-}y  \meets U \right),
  \end{align*}
  $\mathcal{E}$ is the class of $R$-biclosed subsets. Thus
  $\mathcal{E}$ has a generating subset $G$  by $\NIDbi$.
  Define a relation $r \subseteq G \times X$ by 
  \[
    U \mathrel{r} x \defequiv x \in U.
  \]
  Clearly, we have $r_{1} \circ r = r_{2} \circ r$.
  We show that $r$ is a weak equaliser of
  $r_{1}$ and~$r_{2}$. Let $s \subseteq Z \times X$ be a relation such that
  $r_{1} \circ s = r_{2} \circ s $.
  Define a relation $\overline{s} \subseteq Z \times G$ by
  \[
    z \mathrel{\overline{s}} U \defequiv U \subseteq sz,
  \]
  where
  $
  sz \defeql \left\{ x \in X \mid 
    z \mathrel{s} x\right\}
  $.
  Obviously, we have $r \circ \overline{s} \subseteq s$. Conversely,
  suppose that $z \mathrel{s} x$. Since $sz$ is an
  $R$-biclosed subset of $X$, there exists a $U \in G$ such that $x
  \in U \subseteq sz$.  Thus $z \mathrel{(r \circ
    \overline{s})} x$, and so $s \subseteq r \circ \overline{s}$.
  Hence $s = r \circ \overline{s}$.
  \smallskip

 \noindent (\ref{prop:NIDbiRelWEq2} $\rightarrow$ \ref{prop:NIDbiRelWEq1})
  Assume that $\Rel$ has weak equalisers, and 
  let $R$ be a set of rules on a set $S$.  Then, $R$ corresponds
  to two relations 
  $r_{1}, r_{2} \subseteq  S \times R$ given by
  \begin{align*}
    x \mathrel{r_{1}} (a,b) &\defequiv x \in a, \\
    x \mathrel{r_{2}} (a,b) &\defequiv x \in b.
  \end{align*}
  Let $r \subseteq E \times S$ be a weak equaliser of $r_{1}$ and
  $r_{2}$ in $\Rel$, and put
  \[
    G \defeql \left\{ re \mid e \in E \right\}.
  \]
  Since $r_{1} \circ r = r_{2} \circ r$, elements of $G$ are
  $R$-biclosed subsets of $S$. We show that $G$ generates
  the class of $R$-biclosed subsets.
  Let $\alpha \subseteq S$ be an arbitrary $R$-biclosed subset,
  and let $z \in \alpha$. Define a relation 
  $r_{\alpha} \subseteq \left\{ * \right\} \times S$ by
  \[
    * \mathrel{r_{\alpha}} x \defequiv x \in \alpha,
  \]
  where $\left\{ * \right\}$ is a fixed one-element set.
  Since $r_{\alpha} * = \alpha$
  and $\alpha$ is $R$-biclosed, we have 
  $r_{1} \circ r_{\alpha} = r_{2} \circ r_{\alpha}$.
  Thus, $r_{\alpha}$ factors through $r$ via some relation
  $\overline{r_{\alpha}} \subseteq \left\{ * \right\} \times E$. 
  Since $* \mathrel{r_{\alpha}} z$, there
  exists an $e \in E$ such that $* \mathrel{\overline{r_{\alpha}}} e$
  and $e \mathrel{r} z$.  Then, for any $y \in re$,  we have $*
  \mathrel{(r \circ \overline{r_{\alpha}})} y$, i.e., $*
  \mathrel{r_{\alpha}} y$, and so $y \in \alpha$. Hence $z \in re
  \subseteq \alpha$.
\end{proof}

\subsection{Equalisers in the category of basic pairs}
\label{sec:EqlBP}
We show that $\NIDn{1}$ is equivalent to the
existence of equalisers in the category of basic pairs described 
in the forthcoming book by Sambin~\cite{Sambin_BP_book}. The result in 
this subsection refines the result by
Ishihara and Kawai \cite[Proposition~3.8]{CompletenessOfBPandCSpa}, where
they showed that the category of basic pairs has coequalisers using
$\SGA$ (cf.\ Remark \ref{rem:RelBPSelfDual}).

\begin{definition}
A \emph{basic pair} is a triple $(X,\Vdash,S)$ where
$X$ and $S$ are sets, and $\Vdash$ is a relation from $X$ to $S$.
A \emph{relation pair} between basic pairs $\mathcal{X}_1 =
(X_1,\Vdash_1,S_1)$ and $\mathcal{X}_2 = (X_2,\Vdash_2,S_2)$ is a pair
$(r,s)$ of relations $r \subseteq X_1 \times X_2$ and $s \subseteq S_1
\times S_2$ such that
$
{\Vdash_2} \circ r = s \circ {\Vdash_1},
$
i.e., the following diagram commutes in $\Rel$:
\[
  \begin{tikzcd}
     S_1 \arrow[from=d, "\Vdash_1"] \arrow[r,"s"]
     &
     S_2 \arrow[from=d,"\Vdash_2"']
     \\
     X_1
     \ar[r,"r"]
     &
     X_2
  \end{tikzcd}
\]
Two relation pairs $(r_1,s_1)$ and $(r_2,s_2)$ between basic pairs
$\mathcal{X}_1$ and $\mathcal{X}_2$ are said to be \emph{equivalent} (or
\emph{equal}) if
\begin{equation}
  \label{eq:RelPairEqual}
  {\Vdash_2} \circ {r_1} = {\Vdash_2} \circ {r_2}.
\end{equation}
In this case, we write ${(r_1,s_1)} \sim {(r_2,s_2)}$.

Basic pairs and  relation pairs with equality defined by
\eqref{eq:RelPairEqual} form a
category $\BP$: the identity on a basic pair $\mathcal{X} = (X,
\Vdash, S)$ is $(\id_{X}, \id_{S})$ and the composition of two
relation pairs $(r, s) \colon \mathcal{X}_{1} \to \mathcal{X}_{2}$ and
$(u,v) \colon \mathcal{X}_{2} \to \mathcal{X}_{3}$ is $(u
\circ r, v \circ s)$, where each component is a relational
composition. It is easy to check that compositions respect
equality of relation pairs;
see \cite[Proposition 2.3]{CompletenessOfBPandCSpa} for more details.
\end{definition}

\begin{proposition}
  \label{prop:RelWEqBPEq}
  The following are equivalent over $\ECST$:
  \begin{enumerate}
    \item\label{prop:RelWEqBPEq1} $\Rel$ has weak equalisers;
    \item\label{prop:RelWEqBPEq2} $\BP$ has equalisers.
  \end{enumerate}
\end{proposition}
\begin{proof}
  (\ref{prop:RelWEqBPEq1} $\rightarrow$ \ref{prop:RelWEqBPEq2}) 
  This follows from a general result on the \emph{Freyd completion of
  categories}
  \cite[Section 2.5 (b)]{GrandisFreydCompletion}. We recall the proof
  for the particular case of
  $\Rel$.  Assume that $\Rel$ has weak equalisers.
  Let $(r_{1}, s_{1}), (r_{2},s_{2}) \colon \mathcal{X}_{1} \to
  \mathcal{X}_{2}$ be relation pairs between
  basic pairs $\mathcal{X}_{1} = (X_{1}, \Vdash_{1}, S_{1})$
  and $\mathcal{X}_{2} = (X_{2}, \Vdash_{2}, S_{2})$.
  Put 
    \begin{align*}
      u_{1} &\defeql { \Vdash_{2} } \circ r_{1},  &
      u_{2} &\defeql { \Vdash_{2} } \circ r_{2}.
    \end{align*}
  Let $e \colon E \to
  X_{1}$ be a weak equaliser of $u_{1}$ and $u_{2}$ in $\Rel$.
  Consider a basic pair $\mathcal{E} = (E, \Vdash_1 \mathrel{\circ} {e},
  S_{1})$. Then $(e, \id_{S_{1}})$ is a relation pair from
  $\mathcal{E}$ to $\mathcal{X}_{1}$, and we have
  $(r_{1},s_{1}) \circ (e, \id_{S_{1}}) 
  \sim (r_{2},s_{2}) \circ (e, \id_{S_{1}})$; see the diagram below.
  \[
    \begin{tikzcd}[sep=huge]
      S_{1}
      \arrow[r,"\id_{S_{1}}"]
      &
      S_{1}
      \arrow[r,shift left, "s_{2}"]
      \arrow[r,shift right, "s_{1}"']
      &
      S_{2} \\
      E
      \arrow[u,"\Vdash_{1} \circ e"]
      \arrow[r,"e"]
      & 
      X_{1}
      \arrow[u,"\Vdash_1"]
      \arrow[r,shift left, "r_{2}"]
      \arrow[r,shift right, "r_{1}"']
      \arrow[ru,shift left, "u_{2}", shorten >=3pt,shorten <=-2pt]
      \arrow[ru,shift right, "u_{1}"',, shorten >=3pt,shorten <=-2pt]
      &
      X_{2}
      \arrow[u,"\Vdash_2"']
    \end{tikzcd}
  \]
  We show that $(e, \id_{S_{1}}) \colon \mathcal{E} \to
  \mathcal{X}_{1}$ is an equaliser of $(r_{1},s_{1})$ and
  $(r_{2},s_{2})$. Let $\mathcal{Z} = (Z, \Vdash, T)$ be a basic
  pair, and let $(u,v) \colon \mathcal{Z} \to \mathcal{X}_{1}$
  be a relation pair such that 
  $(r_{1},s_{1}) \circ (u,v) \sim (r_{2},s_{2}) \circ (u,v)$.
  Then, $u_{1} \circ u = u_{2} \circ u$ in $\Rel$, so there exists
  a relation $\overline{u} \subseteq Z \times E$ such that $e \circ
  \overline{u} = u$. Then, $(\overline{u}, v)$ is a relation
  pair from $\mathcal{Z}$ to $\mathcal{E}$. It is also straightforward
  to check that $(e, \id_{S_{1}}) \circ (\overline{u}, v) \sim  (u,v)$
  and that $(\overline{u},v)$ is a unique relation pair from
  $\mathcal{Z}$ to $\mathcal{E}$ with this property (cf.\ the diagram
  below).
  \[
    \begin{tikzcd}[row sep=scriptsize,column sep=large]
      T
      \arrow[rr, "v"]
      \arrow[rd, "v"']
      &&
      S_{1}
      \\[-8pt]
      &
      S_{1} 
      \arrow[ru, "\id_{S_{1}}"',shorten >=3pt,shorten <=-2pt,pos=0.4]
      &
      \\
      Z
      \arrow[uu,"\Vdash"]
      \arrow[rr,"u", near start]
      \arrow[rd,"\overline{u}"']
      && 
      X_{1}
      \arrow[uu,"\Vdash_{1}"']
      \\[-8pt]
      &
      E
      \arrow[ru, "e"',shorten >=3pt,shorten <=-2pt]
      \arrow[uu,"\Vdash_{1} \circ e"',crossing over,pos=0.65]
      &
    \end{tikzcd}
  \]

  \smallskip
  \noindent (\ref{prop:RelWEqBPEq2} $\rightarrow$ \ref{prop:RelWEqBPEq1})
  In the following, we write $S_{\Delta}$ 
  for the basic pair $(S, \id_{S}, S)$ on a set $S$.
  Assume that $\BP$ has equalisers, and 
  let $r_{1},r_{2} \subseteq X \times Y$ be a parallel pair of
  relations.
  Then, $(r_{1},r_{1})$ and $(r_{2},r_{2})$ are relation
  pairs from  $X_{\Delta}$ to $Y_{\Delta}$.
  Let $(r,s) \colon \mathcal{E} \to
  X_{\Delta}$ be an equaliser of $(r_{1}, r_{1})$ and 
  $(r_{2}, r_{2})$ in $\BP$, and write $\mathcal{E} = (E,\Vdash, S)$.
  We show that $r$ is a weak equaliser of $r_{1}$ and
  $r_{2}$ in $\Rel$. First, since $(r_{1},r_{1}) \circ (r,s) \sim
  (r_{2},r_{2}) \circ (r,s) $, we have $r_{1} \circ r = r_{2} \circ r$.
  Next, let $u \subseteq Z \times X$ be a relation such that $r_{1} \circ u
  = r_{2} \circ u$. Then, $(u,u)$ is a relation pair from
  $Z_{\Delta}$ to $X_{\Delta}$, and we have  $(r_{1}, r_{1}) \circ
  (u,u) \sim (r_{2},r_{2}) \circ (u,u)$. Thus, there exists a unique
  relation pair $(v,w) \colon  Z_{\Delta} \to \mathcal{E}$ such
  that $(r,s) \circ (v,w) \sim (u,u)$. Then, $r \circ v = u$.
\end{proof}

\begin{remark}\label{rem:RelBPSelfDual}
The categories $\Rel$ and $\BP$ are self dual, i.e., $\Rel$ is
equivalent to its opposite $\Rel^{\op}$ (and similarly for $\BP$).
Thus, $\Rel$ has weak equalisers if and only if it has weak coequalisers,
and $\BP$ has equalisers if and only if it has  coequalisers.
Moreover, since $\Rel$ has small products and hence small coproducts as
well, $\BP$ has small products and coproducts~\cite[Section 2.5
(a)]{GrandisFreydCompletion} (see also
\cite[Proposition~3.2]{CompletenessOfBPandCSpa}).
Hence, the following are equivalent over $\ECST$:
  \begin{enumerate}
    \item $\BP$ has (co)equalisers;
    \item $\BP$ is (co)complete.
  \end{enumerate}
\end{remark}

We summarise the equivalents of $\NIDn{1}$.
\begin{theorem}
  \label{thm:NID1}
The following are equivalent over $\ECST$:
\begin{enumerate}
\item $\NIDn{1}$;
\item $\NIDbi$;
\item $\Rel$ has weak (co)equalisers;
\item $\BP$ has (co)equalisers;
\item $\BP$ is complete and cocomplete.
\end{enumerate}
\end{theorem}

%%%%%%%%%%%%%%%%%%%%%%%%%%%%%
\section{Finitary $\NID$} \label{sec:NIDf}
%%%%%%%%%%%%%%%%%%%%%%%%%%%%%
In this section, we give some statements equivalent to $\NIDf$.

\subsection{Models of geometric theories}
\label{sec:SGM}
The connection between $\NID$ principles and set-generation of
the class of models of a game theory was studied by van den Berg
\cite[Section 4]{vandenBergNID}. In particular, he showed that
$\NIDf$ is equivalent to the statement that the class
of models of any finitary game theory is set-generated~\cite[Corollary
4.4]{vandenBergNID}.  Since geometric theories form a subclass of
finitary game theories, the result in this subsection follows from
his result. However, we provide a small refinement, which serves as
a stepping-stone for Section \ref{sec:n-aryNID}.

\begin{definition}
  A (propositional) \emph{geometric
  theory} over a set $S$ is a set of \emph{axioms} of the  form
  \[
    \medwedge A \vdash \bigvee_{i \in I} \medwedge B_i
  \]
  where $I$ is a set and $A, B_i$ are finitely enumerable subsets of 
  $S$. 
    If $T$ is a geometric theory over $S$, a \emph{model} of $T$ is
    a subset $m \subseteq S$ such that
  \[
    A \subseteq m \implies  \exists i \in I \left(B_i \subseteq
    m \right)
  \]
  for each axiom $\medwedge A \vdash \bigvee_{i \in I} \medwedge B_i$
  in $T$.  The class of models of $T$ is denoted by $\Model{T}$.
\end{definition}

\begin{definition}
Let $\NIDn{\leq 2}$ be the principle obtained from $\NID$ by
restricting the set $R$ to those rules
$(a,b)$ where $a$ is a surjective image of
$\{0,\dots,n-1\}$ for some  $n \leq 2$.  %$a$ is of the form $\left\{ x,y \right\}$.% (cf.\ Definition \ref{def:n-ary}).
\end{definition}
\begin{proposition}
  \label{prop:NIDfSGM}
  The following are equivalent over $\ECST$:
  \begin{enumerate}
    \item\label{prop:NIDfSGM1} $\NIDn{\leq 2}$;
    \item\label{prop:NIDfSGM2} $\NIDf$;
    \item\label{prop:NIDfSGM3} The class of models of any geometric theory is set-generated.
  \end{enumerate}
\end{proposition}
\begin{proof}
  Clearly \ref{prop:NIDfSGM2} implies \ref{prop:NIDfSGM1}.
  The equivalence of \ref{prop:NIDfSGM2} and \ref{prop:NIDfSGM3} is a corollary of van den Berg \cite[Corollary
  4.4]{vandenBergNID}. We  give a proof for the sake of
  completeness. 
  \smallskip

  \noindent (\ref{prop:NIDfSGM1} $\rightarrow$ \ref{prop:NIDfSGM3}) 
  Assume $\NIDn{\leq 2}$. Let $T$ be a geometric theory over
  a set $S$. Define a set $R$ of rules
  on $\fin(S)$ by
  \begin{align}
    R \defeql&
    \left\{ ( \emptyset, \left\{ \emptyset \right\} ) \right\}
    \label{R:empty} \\
    &\cup \left\{ (\left\{ A \right\}, \left\{ B \right\}) \mid B \subseteq A \right\} 
    \label{R:order}\\
   &\cup \left\{ (\left\{ A,B \right\}, \left\{ A \cup B \right\}) \mid
   A,B \in \fin(S) \right\}
    \label{R:join} \\
  &\cup \Bigl\{ \left(\left\{ A \right\}, \left\{ B_{i} \mid i \in I
\right\}\right)
  \mid \medwedge A \vdash \bigvee_{i \in I}\medwedge B_{i} \in T\Bigr\}. 
  \label{R:axiom}
  \end{align}
  Note that $\fin(S)$ is a set since $\NIDn{\leq 2}$ implies $\NIDn{0}$
  (cf.\ Remark \ref{rem:NID0Fullness}).
  Note also that $R$ consists of nullary and binary rules.

  Let $\mathcal{C}$ be the class of $R$-closed subsets. Define functions
  $\Phi \colon \mathcal{C} \to \Model{T}$
  and
  $\Psi \colon \Model{T} \to  \mathcal{C}$
  by 
  \begin{align*}
    \Phi(\alpha) &\defeql  \mathop{\cup} \alpha, &
    \Psi(m) &\defeql \fin(m).
  \end{align*}
  It is straightforward to show that these functions are well-defined.
  We show that they are inverses of each other. First,
  we have 
  \[
    \Phi(\Psi(m)) = \mathop{\cup} \Psi(m) = \mathop{\cup} \fin(m) = m
  \]
 for each $m \in \Model{T}$.
  Next, for each $\alpha \in \mathcal{C}$, we have
  $\alpha \subseteq \fin(\mathop{\cup} \alpha) = \Psi(\Phi(\alpha))$. Conversely, let $A \in
  \fin(\mathop{\cup} \alpha)$, and write $A = \left\{ x_{0},\dots,x_{n-1} \right\}$.
 For each $i<n$, there is a $B_i \in \alpha$ such that $x_i \in B_i$,
 so $\{x_i\} \in \alpha$ by \eqref{R:order}.
 Then $A=\left\{ x_0,\dots,x_{n-1} \right\} \in \alpha$
 by \eqref{R:join} and induction.
  Note that if $A =
  \emptyset$, then $A \in \alpha$ by \eqref{R:empty}.
  Thus, $\Psi(\Phi(\alpha)) = \alpha$ for each $\alpha \in \mathcal{C}$.
  By $\NIDn{\leq 2}$,  the class $\mathcal{C}$ has a generating subset
  $G$. Then, $\left\{ \Phi(\alpha) \mid \alpha \in G\right\}$ is a
  generating subset of $\Model{T}$.
  \smallskip

  \noindent(\ref{prop:NIDfSGM3} $\rightarrow$ \ref{prop:NIDfSGM2})
  Assume that the class of models of any geometric theory is
  set-generated, and 
  let $R$ be a set of finitary rules on a set $S$.
  We can identify each rule $(a, b)$  in $R$ with
  the following geometric axiom over $S$:
\[
  \medwedge a \vdash  \bigvee_{y \in b} y.
\]
Let $T_{R}$ be the geometric theory over $S$ with the
axioms of the above form for each rule  in $R$. Then, a model of
$T_{R}$ is just an $R$-closed subset of $S$.
Hence, the class of $R$-closed subsets is set-generated.
\end{proof}
\subsection{$n$-ary $\NID$}
\label{sec:n-aryNID}
By setting a uniform bound on the size of the premise of finitary
rules, we have countably many fragments of $\NIDf$.
\begin{definition}
  \label{def:n-ary}
Let $n \in \omega$.
A rule $(a, b)$ on a set $S$ is said to be  \emph{$n$-ary} if there exists a
surjection $f\colon \left\{ 0,\dots,n-1 \right\} \to a$.
The $n$-ary $\NID$, denoted by $\NIDn{n}$, is
the principle obtained from $\NID$ by restricting the set 
$R$ to $n$-ary rules.
\end{definition}
  
\begin{lemma}
  \label{lem:NIDleq2EquivNID2}
  $\NIDn{\leq 2}$ and $\NIDn{2}$ are equivalent
  over $\ECST$.
\end{lemma}
\begin{proof}
  It suffices to show that $\NIDn{2}$ implies $\NIDn{\leq 2}$. Assume
  $\NIDn{2}$, and let $R$ be a set of rules on $S$ consisting of
  nullary and binary rules. 
  Choose any set $*_{S}$ not in $S$, and 
  define a set  $R'$ 
  of binary rules on $S \cup \left\{ *_{S} \right\}$ by 
  \begin{align*}
    R' &\defeql
    \left\{ (\left\{ *_{S} \right\}, b) \mid (\emptyset, b) \in R
  \right\}
    \cup
    \left\{ (a, b) \in R \mid a \meets a \right\}.
  \end{align*}
  By $\NIDn{2}$, the class of $R'$-closed subsets of $S \cup \left\{
  *_{S} \right\}$ has a generating subset $G$. Put
  \[
    H \defeql \left\{ \alpha \cap S \mid \alpha \in G,
      *_{S} \in \alpha \right\}.
  \]
  We show that $H$ generates the class of $R$-closed subsets of $S$.
  First, let $\alpha \in G$ such that $*_{S} \in \alpha$.
  Consider any $(a,b) \in R$ such that $a \subseteq \alpha \cap S$.
  If $a = \emptyset$, then $b \meets (\alpha \cap S)$ because
  $*_{S} \in \alpha$.
  If $a$ is inhabited, then obviously $b \meets (\alpha \cap S)$. 
  Thus, the elements of $H$ are $R$-closed.
  Let $\beta$ be any $R$-closed subset of $S$, and let $x \in \beta$.
  Then, $\beta \cup \left\{ *_{S} \right\}$ is an $R'$-closed subset of
  $S \cup \left\{ *_{S} \right\}$. Thus, there exists an $\alpha \in
  G$ such that $x \in \alpha \subseteq \beta \cup \left\{ *_{S}
  \right\}$. Then, $x \in \alpha \cap S \subseteq \beta$. 
\end{proof}
Lemma \ref{lem:NIDleq2EquivNID2}, together with Proposition
\ref{prop:NIDfSGM},
yields the following.
\begin{proposition}
  The following are equivalent over $\ECST$:
  \begin{enumerate}
    \item $\NIDf$;
    \item $\NIDn{n} \;(n \geq 2)$.
  \end{enumerate}
\end{proposition}

\subsection{Formal points of formal topologies}
\label{sec:PTSGFTop}
The initial motivation of the $\NID$ principle comes from the problem
of constructing generating subsets in
constructive point-free topology, where
some of its results require various extensions of $\CZF$.
Van den Berg~\cite[Section 5]{vandenBergNID}
and Aczel et al.\ \cite[Section 7.2]{SGA} illustrated the power of
$\NID$ by providing a uniform solution to these problems using $\NID$
and $\SGA$, respectively. With some adjustments to the setting
of $\ECST$, we turn some of
their results into equivalents of $\NIDf$.

We adopt the following definition of formal
topology~\cite{Sambin:intuitionistic_formal_space}, which is the notion of
point-free topology in constructive and predicative foundations.
\begin{definition}[{Coquand et al.\  \cite[Definition
2.1]{CoquandIndGenFTop}}]
  A \emph{formal topology} is a triple $\mathcal{S} =
  (S, \leq, \cov)$ where $(S, \leq)$ is a preordered set and 
   $\cov$ is a relation from $S$ to  $\pow(S)$
   such that
   $
   \left\{ a \in S \mid a \cov U \right\}
   $
  is a set for each $U \subseteq S$ and 
  \begin{enumerate}
    \item $a \in U \implies a \cov U$,
    \item $a \cov U \amp  U \cov V \implies a \cov V$,
    \item $a \cov U \amp a \cov V \implies a \cov  U \downarrow V $,
    \item $a \leq b \implies a \cov \left\{ b \right\}$,
  \end{enumerate}
  where
  \[
    \begin{aligned}
      U \cov V &\defequiv \forall a \in U \left( a \cov V\right), \\
      U \downarrow V &\defeql
      \left\{ c \in S \mid  \exists a \in U  \exists
      b \in V \left( c \leq a \wedge c \leq b \right) \right\}.
    \end{aligned}
  \]
  A subset $\alpha \subseteq S$ is a \emph{formal point} of $\mathcal{S}$ if
  \begin{enumerate}[({P}1)]
    \item $\alpha$ is inhabited,
    \item $ a, b \in \alpha \implies \alpha \meets (a \downarrow
      b)$,
    \item\label{P3} $ a \in \alpha \amp a \cov U \implies \alpha
      \meets U$,
  \end{enumerate}
  where $a \downarrow b \defeql \left\{ a \right\} \downarrow \left\{
  b \right\}$.
  The class of formal points of $\mathcal{S}$ is denoted by
  $\Pt(\mathcal{S})$.
\end{definition}

Our main interest is in inductively generated topologies, which  allow
us to reason about formal topologies using selected sets of axioms.
\begin{definition}
  An \emph{axiom-set} on a set $S$ is a pair $(I,C)$, where
  $(I(a))_{a \in S}$ is a family of sets indexed by  $S$, and $C$ is a
  family $(C(a,i))_{a \in S, i \in I(a)}$ of subsets  of $S$
  indexed by $\sum_{a \in S}I(a)$.
  A formal topology $(S,\leq, \cov)$ is \emph{inductively generated}
  by $(I,C)$ if $\cov$ is the smallest among the relations $\cov'$ such that
  \begin{enumerate}
    \item $a \leq b \cov' U \implies a \cov' U$,
    \item $a \cov' C(a,i)$ for each $i \in I(a)$,
  \end{enumerate}
  and which makes $(S,\leq, \cov')$ a formal topology.
\end{definition}

A formal point of an inductively generated formal
topology can be characterised by
an axiom-set, where condition \ref{P3} is
replaced by
\begin{enumerate}
  \item $a \leq b \amp a \in \alpha \implies b \in \alpha$,
  \item $a \in \alpha \implies \alpha \meets C(a,i)$ for each
    $i \in I(a)$.
\end{enumerate}

\begin{remark}
  The construction of an inductively generated formal topology requires
  $\CZF$ extended with the Regular Extension Axiom
  \cite{Aczel-Rathjen-Note}, which is much stronger than $\ECST$. 
  However, in Proposition \ref{prop:NIDfPt} and 
  Proposition \ref{prop:LKFTopMap}, all we need is a preorder 
  equipped with an axiom-set. Hence, in this paper, we identify inductively
  generated formal topologies with preorders equipped with axiom-sets,
  and adopt the notions of formal point and formal topology map
  formulated in terms of axiom-sets.
\end{remark}

\begin{proposition}
  \label{prop:NIDfPt}
  The following are equivalent over $\ECST + \FPA$:
  \begin{enumerate}
    \item\label{prop:NIDfPt1} $\NIDf$;
    \item\label{prop:NIDfPt2} The class of formal points of any inductively generated formal
    topology is set-generated.
  \end{enumerate}
\end{proposition}
\begin{proof}
  (\ref{prop:NIDfPt1} $\rightarrow$ \ref{prop:NIDfPt2})
  Assume $\NIDf$. Let $\mathcal{S} = (S,\leq,\cov)$ be a formal topology inductively
   generated by an axiom-set $(I,C)$ on $S$. Then, formal points of
   $\mathcal{S}$ are closed subsets of the following set of finitary
   rules on $S$:
   \[
     \begin{aligned}
       &\left\{ (\emptyset, S) \right\} \\
       &\cup
       \left\{ (\{ a,b \},  a \downarrow b) \right\} \\
       &\cup
       \left\{ (\{ a \}, \left\{ b \right\}) \mid a \leq b \right\} \\
       &\cup
       \left\{ (\{ a \}, C(a,i)) \mid a \in S, i \in
     I(a) \right\}.
   \end{aligned}
 \]
   Thus $\Pt(\mathcal{S})$ is set-generated.
  \smallskip

  \noindent
  (\ref{prop:NIDfPt2} $\rightarrow$ \ref{prop:NIDfPt1})
  Assume that the class of formal points of any inductively generated
  formal topology is set-generated. Let $R$ be a set of finitary rules
  on a set $S$. Using $\FPA$, define an axiom-set $(I,C)$ on $\fin(S)$ by 
  \begin{align*}
    I(a) &\defeql \left\{ (b,c) \in R \mid b = a\right\}, \\
    C(a,(b,c)) &\defeql \left\{ \left\{ y \right\} \mid y \in c
  \right\}.
  \end{align*}
  Let $\mathcal{S} = (\fin(S), \supseteq, \cov)$ be a formal
  topology inductively generated by $(I,C)$ using the reverse
  inclusion order on $\fin(S)$, 
  and let $\mathcal{C}$ be the class of $R$-closed subsets of $S$.
  As in the proof of (\ref{prop:NIDfSGM1} $\rightarrow$ \ref{prop:NIDfSGM3}) in
  Proposition~\ref{prop:NIDfSGM}, one can show that the mappings
  \begin{align*}
    \alpha &\mapsto \mathop{\cup} \alpha \colon \Pt(\mathcal{S}) \to
    \mathcal{C}, &
    \beta &\mapsto \fin(\beta) \colon \mathcal{C} \to
    \Pt(\mathcal{S})
  \end{align*}
  are well-defined and that they are inverses of each other.
  By the assumption, $\Pt(\mathcal{S})$ has a generating subset
  $G$. Then, $H = \left\{ \mathop{\cup} \alpha \mid \alpha \in G\right\}$
  is a generating subset of the class of $R$-closed subsets of $S$.

  Note that $\FPA$ is needed only in the direction (\ref{prop:NIDfPt2}
  $\rightarrow$ \ref{prop:NIDfPt1}).
\end{proof}
A formal point is an instance of morphisms between formal
topologies.
\begin{definition}\label{def:FTM}
  Let $\mathcal{S} = (S, \leq, \cov)$ and $\mathcal{S}'  = (S', \leq',
  \cov')$ be formal topologies.  A relation $r \subseteq S \times S'$
  is a \emph{formal topology map} from $\mathcal{S}$ to $\mathcal{S}'$ if 
  \begin{enumerate}[({FTM}1)]
    \item\label{FTM1} $S\cov r^{-}S'$,
    \item\label{FTM2} $r^{-}a \downarrow r^{-}b \cov r^{-}(a \downarrow' b)$,
    \item\label{FTM3} $a \cov' U \implies r^{-}a \cov r^{-}U$.
  \end{enumerate}
  Two formal topology maps $r,s \colon \mathcal{S} \to\mathcal{S}'$
  are \emph{equal} if
  $
  r^{-}a \cov s^{-} a
  $
  and
  $
  s^{-}a \cov r^{-} a
  $
  for all $a \in S'$.

  If $r \colon \mathcal{S} \to \mathcal{S}'$ is a formal topology map
  and $\mathcal{S}'$ is inductively generated by an
  axiom-set $(I,C)$ on $S'$, then the condition \ref{FTM3} can be replaced by
  \begin{enumerate}[({FTM3}a)]
    \item\label{FTM3a} $a \leq' b \implies r^{-} a \cov r^{-} b$,
    \item\label{FTM3b} $r^{-} a \cov r^{-}C(a,i)$ for each $i \in
      I(a)$.
  \end{enumerate}
\end{definition}
\begin{remark}\label{rem:PtMap}
Let $\One$ denote the formal topology $(\left\{ * \right\}, =,
\in)$. Then, a formal point $\alpha$ of a formal topology
$\mathcal{S}$ corresponds to a formal topology map
$r_{\alpha} \colon \One \to \mathcal{S}$ given by 
  $
  * \mathrel{r_{\alpha}} a \defequiv a \in \alpha.
  $
\end{remark}

Set-presentations of formal topologies provide a stronger notion
of inductive generation.
\begin{definition}
  A formal topology $\mathcal{S} = (S,\leq, \cov)$ is \emph{set-presented} if there
  exists an axiom-set $(I,C)$ on $S$ such that
  \[
    a \cov U \iff \exists i \in I(a) \left( C(a,i) \subseteq U \right).
  \]
  In this case, $(I,C)$ is called a \emph{set-presentation} of
  $\mathcal{S}$.
\end{definition}

\begin{proposition}
  \label{prop:LKFTopMap}
  The following are equivalent over $\ECST + \FPA$:
  \begin{enumerate}
    \item\label{prop:LKFTopMap1} $\NIDf$;
    \item\label{prop:LKFTopMap2} The class of formal topology maps from a set-presented
      formal topology to an inductively generated formal
    topology is set-generated.
  \end{enumerate}
\end{proposition}
\begin{proof}
  Since $\One$ is set-presented, \ref{prop:LKFTopMap2} implies 
  \ref{prop:LKFTopMap1} by Proposition~\ref{prop:NIDfPt}
  and Remark~\ref{rem:PtMap}. Note that we need $\FPA$ in this
  direction.

  Conversely, assume $\NIDf$.
  Let $\mathcal{S} = (S, \leq, \cov)$ be a 
  formal topology with a set-presentation $(I,C)$,
  and let $\mathcal{T} = (T, \leq', \cov')$ be
  a formal topology inductively generated by an axiom-set $(J,D)$.
  Then, a formal topology map $r \colon \mathcal{S} \to \mathcal{T}$ is
  a model of the following generalised geometric theory over $S \times T$
  (cf.\ Aczel et al.\ \cite[Proposition 7.8]{SGA}):
  \begin{align}
  \label{GGA-FTM1}
    &\vdash \bigvee_{i \in I(a)}\bigwedge_{a' \in
    C(a,i)}\bigvee_{b \in T}(a',b) 
    &&(a \in S) \\
  \label{GGA-FTM2}
    (b',b) \wedge (c',c) &\vdash \bigvee_{i \in I(a)}\bigwedge_{a' \in
    C(a,i)}\bigvee_{d \in b \downarrow' c} (a',d)
    && (a \in b' \downarrow c')\\
  \label{GGA-FTM3}
    (a,b) &\vdash \bigvee_{i \in I(a)}\bigwedge
    C(a,i) \times \left\{ c \right\}
    && (b \leq' c)\\
  \label{GGA-FTM4}
    (a,b) &\vdash \bigvee_{i \in I(a)}\bigwedge_{a' \in
    C(a,i)}\bigvee_{b' \in D(b,j)} (a',b')
    && (j \in J(b))
  \end{align}
  where \eqref{GGA-FTM1}, \eqref{GGA-FTM2}, \eqref{GGA-FTM3},
  and \eqref{GGA-FTM4} are derived from 
  \ref{FTM1}, \ref{FTM2}, \ref{FTM3a},  and \ref{FTM3b},
  respectively, using the fact that $\mathcal{S}$ is set-presented by $(I,C)$.
  The disjunctions such as $\bigvee_{b \in T}(a',b)$ must be read as
  $\bigvee_{b \in T} \bigwedge \left\{ (a',b) \right\}$.
  Then, the required conclusion follows from Remark \ref{rem:Rankn}.
\end{proof}

\subsection{Equalisers in the category of concrete spaces}
\label{sec:EqualiserCSpa}
We show that $\NIDf$ is equivalent to the existence of equalisers in
the category of concrete spaces, a predicative notion of point-set
topology by Sambin~\cite{Sambin_BP_book}. 
Ishihara and Kawai \cite[Section 4]{CompletenessOfBPandCSpa}
have already shown that the category is complete and cocomplete using $\SGA$.
Hence, the essence of this subsection is that the converse holds.
\begin{definition}
A \emph{concrete space} is  a basic pair $(X,\Vdash,S)$ such that
\begin{enumerate}
  \item  $X = \ext S$,
  \item  $\ext a \cap \ext b = \ext(a \downarrow b)$
\end{enumerate}
for all $a,b \in S$, where 
\begin{gather}
  \label{def:ext}
  \ext a \defeql \left\{ x \in X \mid x \Vdash a \right\},
  \qquad
  \qquad
  \ext U \defeql  \bigcup_{a \in U} \ext a,\\
  \notag
  a \downarrow b \defeql \left\{ c \in S \mid \ext c \subseteq \ext a
  \cap \ext b \right\}.
\end{gather}
We also define
  $
  U \downarrow V \defeql \bigcup_{a \in U, b \in V} a \downarrow b
  $
for $U,V \in \pow(S)$.

Let $\mathcal{X}_{1} = (X_{1},\Vdash_{1},S_{1})$
and 
$\mathcal{X}_{2} = (X_{2},\Vdash_{2},S_{2})$ be basic pairs.
A relation pair $(r,s) \colon \mathcal{X}_{1} \to
\mathcal{X}_{2}$ is said to be \emph{convergent} if 
\begin{enumerate}
  \item  $\ext_{1} S_{1} = r^{-}\ext_{2} S_{2}$,
  \item  $\ext_{1}(s^{-}a \downarrow s^{-}b) = r^{-}\ext_{2}(a
    \downarrow b)$,
\end{enumerate}
where $\ext_{i} \;(i = 1,2)$ is the operator
given by \eqref{def:ext}
associated with $\Vdash_i \,(i = 1,2)$.
\end{definition}

Concrete spaces and convergent relation pairs form a
subcategory $\CSpa$ of $\BP$.
Specifically, $\CSpa$ is a coreflective subcategory of $\BP$.
\begin{proposition}[{Ishihara and Kawai~\cite[Theorem
  7.1]{CompletenessOfBPandCSpa}}]\label{prop:Coreflection}
  For any basic pair $\mathcal{X}$, there exist a concrete space
  $\widetilde{\mathcal{X}}$ and a relation pair $(r,s) \colon
  \widetilde{\mathcal{X}}
  \to\mathcal{X}$  such that for any concrete space
  $\mathcal{Y}$ and a relation pair $(u,v) \colon \mathcal{Y} \to
  \mathcal{X}$ there exists a unique convergent relation pair
  $(\widetilde{u},\hat{v}) \colon \mathcal{Y} \to
  \widetilde{\mathcal{X}}$ which makes the following diagram commute:
  \begin{equation}\label{eq:BPCSpaExtension}
    \begin{tikzcd}
      \mathcal{Y} 
      \arrow[r,"{(\widetilde{u},\hat{v})}"]
      \arrow[dr,"{(u,v)}"']
      &
      \widetilde{\mathcal{X}}
      \arrow[d,"{(r,s)}",pos=0.4]
      \\
      & \mathcal{X}
    \end{tikzcd}
  \end{equation}
\end{proposition}
\begin{proof}
  See Ishihara and Kawai~\cite[Theorem
  7.1]{CompletenessOfBPandCSpa}
  for the details. Their proof can be carried out in $\ECST + \FPA$.
\end{proof}

The construction of an equaliser of $\CSpa$ uses the notion of
convergent subset.
\begin{definition}
  Let $\mathcal{X} = (X, \Vdash, S)$ be a basic pair. A subset $D
  \subseteq X$ is
  said to be \emph{convergent} if 
  \begin{enumerate}
      \item $D \meets \ext S$,
      \item  $D \meets \ext  a  \amp D \meets \ext b
      \implies  D \meets \ext(a \downarrow b)$.
  \end{enumerate}
  The class of   convergent subsets of
  a basic pair $\mathcal{X}$ is denoted by $\Conv(\mathcal{X})$.
\end{definition}
An equaliser in $\CSpa$ is constructed from a generating subset
of a certain class.
In the lemma below, $\Diamond_2 U$ denotes the set
$\left\{ a \in S_{2} \mid  \exists x \in U 
\left(x \Vdash_{2} a\right)  \right\}$
for each subset $U \subseteq X_2$.
\begin{lemma}[{Ishihara and Kawai \cite[Proposition 6.3]{CompletenessOfBPandCSpa}}]
  \label{lem:CSpaEql}
  Let $\mathcal{X}_{1} = (X_{1},\Vdash_{1},S_{1})$ and $\mathcal{X}_{2}
  = (X_{2},\Vdash_{2},S_{2})$ be concrete spaces and $(r_{1}, s_{1}),
  (r_{2}, s_{2}) \colon \mathcal{X}_{1} \to \mathcal{X}_{2}$
  be a parallel pair of morphisms in $\CSpa$.
  If the class defined by
  \begin{equation}
    \label{eq:CSpaEql}
    \mathcal{E} \defeql \left\{ D \in \Conv(\mathcal{X}_1) \mid \Diamond_{2}
    r_{1}D \;=\; \Diamond_{2} r_{2}D \right\}
  \end{equation}
  is set-generated, then the parallel pair has an equaliser.
\end{lemma}
\begin{proof}
  See the proof of Ishihara and Kawai~\cite[Proposition
  6.3]{CompletenessOfBPandCSpa}. The proof can be carried out in $\ECST$.
\end{proof}

\begin{proposition}
  \label{prop:EqlImpSGMod}
  The following are equivalent over $\ECST + \FPA$:
  \begin{enumerate}
    \item\label{prop:EqlImpSGMod1} $\NIDf$;
    \item\label{prop:EqlImpSGMod2} $\CSpa$ has equalisers.
  \end{enumerate}
\end{proposition}
\begin{proof}
  (\ref{prop:EqlImpSGMod1} $\rightarrow$ \ref{prop:EqlImpSGMod2})
  This has been proved by Ishihara and Kawai~\cite[Proposition
  6.3]{CompletenessOfBPandCSpa} using $\SGA$.
  We reformulate the proof using $\NIDf$. Assume $\NIDf$, and let
  $(r_{1}, s_{1}), (r_{2}, s_{2}) \colon \mathcal{X}_{1} \to
  \mathcal{X}_{2}$
  be a parallel pair of morphisms in $\CSpa$.
  It suffices to show that the class $\mathcal{E}$ given in
  \eqref{eq:CSpaEql} is set-generated. Define a set
   of finitary rules on $X_1$ by 
   \[
     \begin{aligned}
       R \defeql&
       \left\{(\emptyset, \ext_{1}S_{1})  \right\} \\
       &\cup
       \left\{ (\{ x,y \}, \ext_{1} (a \downarrow b)) \mid
       x \Vdash_{1} a, y \Vdash_{1} b \right\} \\
     &\cup
     \left\{ (\{ x \}, r_{2}^{-}\ext_{2} c ) \mid
     c \in \Diamond_{2} r_{1} \{ x \} \right\} \\
     &\cup
     \left\{ (\{ x \}, r_{1}^{-}\ext_{2} c ) \mid
     c \in \Diamond_{2} r_{2} \{ x \} \right\}.
   \end{aligned}
 \]
  Then, $\mathcal{E}$ is the class of $R$-closed subsets of
  $X_{1}$, and thus it is set-generated.
  \smallskip

  \noindent
  (\ref{prop:EqlImpSGMod2} $\rightarrow$ \ref{prop:EqlImpSGMod1})
  Assume that $\CSpa$ has equalisers.  
  In the following, we write $S_{\Delta}$ for the basic pair $(S, \id_{S}, S)$ on a set $S$ and 
$\fin(S)_{\supseteq}$ for the concrete space $(\fin(S), \supseteq, \fin(S))$. 

  Let $R$ be a set of finitary rules on a set $S$.
  Define relations $r_{1}, r_{2}
  \subseteq \fin(S) \times R$ by
  \[
    \begin{aligned}
      A \mathrel{r_{1}} (a,b)
      &\defequiv A \supseteq a, \\
      A \mathrel{r_{2}} (a,b)
      &\defequiv
      \exists y \in b \left(  A \supseteq \left\{ y \right\} \cup a \right).
    \end{aligned}
  \]
  Obviously, $(r_{1},r_{1})$ and $(r_{2},r_{2})$ are relation pairs from
  $\fin(S)_{\supseteq}$ to $R_{\Delta}$. By
  Proposition~\ref{prop:Coreflection}, $(r_{1},r_{1})$ and $(r_{2},r_{2})$
  determine unique convergent relation pairs
  $(\widetilde{r_{1}},\widehat{r_{1}})$ and $(\widetilde{r_{2}},\widehat{r_{2}})$
  from $\fin(S)_{\supseteq}$ to $\widetilde{R_{\Delta}}$,
  respectively, which make the diagram \eqref{eq:BPCSpaExtension}
  commute.
  Let $(p,q) \colon \mathcal{X} \to \fin(S)_{\supseteq}$  be an equaliser of
  $(\widetilde{r_{1}},\widehat{r_{1}})$ and
  $(\widetilde{r_{2}},\widehat{r_{2}})$ in $\CSpa$, and write
  $\mathcal{X} = (X, \Vdash, K)$. Define a set $G$ of subsets of $S$ by 
  \begin{align*}
    G &\defeql \left\{ \alpha_{x} \mid x \in X\right\},
    \shortintertext{where}
    \alpha_{x}
    &\defeql 
    \left\{ z \in S \mid
      \exists A \in \fin(S) \left( x \mathrel{p} A \wedge A
      \supseteq \left\{ z \right\}\right) \right\}.
    \end{align*}
  Let $x \in X$ and $(a,b) \in R$ such that $a \subseteq \alpha_{x}$. Since
  $(p,q)$ is convergent, there exists a $B \in \fin(S)$ such that
  $x \mathrel{p} B$ and $B \supseteq a$. Thus, $x \mathrel{(r_{1}
\circ p)} (a,b)$.
  Since $(r_{1}, r_{1}) \circ (p,q) \sim (r_{2}, r_{2}) \circ
  (p,q)$, we have $x \mathrel{(r_{2} \circ p)} (a,b)$, so there
  exist $C \in \fin(S)$ and $y \in b$ such that $x
  \mathrel{p} C$ and $C \supseteq \left\{ y \right\} \cup a$. Then
  $y \in \alpha_{x}$, and so $\alpha_{x}$ is $R$-closed.
  
  We show that $G$ generates the class of $R$-closed subsets of
  $S$.
  Let $\beta$ be an arbitrary $R$-closed subset, and $z \in
  \beta$. 
  Define a relation pair
  $(p_{\beta}, q_{\beta}) \colon \left\{ * \right\}_{\Delta} \to \fin(S)_{\supseteq}$
  by
  \[
    \begin{aligned}
      * \mathrel{p_{\beta}} A &\defequiv A \subseteq \beta, \\
      * \mathrel{q_{\beta}} A &\defequiv A \subseteq \beta.
    \end{aligned}
  \]
  It is easy to see that $(p_{\beta}, q_{\beta})$ is convergent.
  Then, for any $(a,b) \in R$
  \begin{align*}
    * \mathrel{(r_{1} \circ p_{\beta})} (a,b) 
    &\iff
     \exists A \in \fin(S) \left( a \subseteq A \subseteq
    \beta \right)\\
    &\iff
    a \subseteq \beta\\
    &\iff
    \exists y \in b \left(  a \cup \left\{ y \right\} \subseteq
    \beta \right) \\
    &\iff
    * \mathrel{(r_{2} \circ p_{\beta})} (a,b).
  \end{align*}
  Thus 
  $({r_{1}},{r_{1}}) \circ (p_{\beta},q_{\beta})
  \sim 
  ({r_{2}},{r_{2}}) \circ (p_{\beta},q_{\beta})$,
  and so
  $(\widetilde{r_{1}},\widehat{r_{1}}) \circ (p_{\beta},q_{\beta})
  \sim 
  (\widetilde{r_{2}},\widehat{r_{2}}) \circ (p_{\beta},q_{\beta})$
  by Proposition~\ref{prop:Coreflection}.
  Hence, $(p_{\beta},q_{\beta})$ factors uniquely through $(p,q)$ via
  a convergent relation pair $(u,v) \colon \left\{ * \right\}_{\Delta} \to \mathcal{X}$ as  follows:
  \[
    \begin{tikzcd}[row sep=scriptsize]
      \left\{ * \right\}
      \arrow[rr, "q_{\beta}"]
      \arrow[rd, "v"']
      &[6]&
      \fin(S)
      \\[-8pt]
      &
      K
      \arrow[ru, "q"']
      &
      \\
      \left\{ * \right\}
      \arrow[uu,"\id_{\left\{ * \right\}}"]
      \arrow[rr,"p_{\beta}", near start]
      \arrow[rd,"u"']
      && 
      \fin(S)
      \arrow[uu,"\supseteq"']
      \\[-8pt]
      &
      X
      \arrow[ru, "p"']
      \arrow[uu,"\Vdash"',crossing over,pos=0.65]
      &
    \end{tikzcd}
  \]
  Since $* \mathrel{p_{\beta}} \left\{ z \right\}$, there exist $x \in X$ and $B \in
  \fin(S)$ such that $* \mathrel{u} x$, $x \mathrel{p} B$, and $B
  \supseteq \left\{ z \right\}$. Thus $z \in \alpha_{x}$. Moreover,  for any
  $y \in \alpha_{x}$, there exists an $A \in \fin(S)$ such that
  $\left\{ y \right\} \subseteq A \subseteq \beta$.
  Hence $\alpha_{x} \subseteq \beta$. 
  Therefore, $G$ generates the class of $R$-closed subsets.

  Note that $\FPA$ is needed only in the direction
  (\ref{prop:EqlImpSGMod2} $\rightarrow$ \ref{prop:EqlImpSGMod1}).
\end{proof}
\begin{remark}
Ishihara and Kawai \cite[Proposition 6.4]{CompletenessOfBPandCSpa}
showed that $\CSpa$ has small products using $\SGA$. Thus, if
$\CSpa$ has equalisers, then $\CSpa$ is complete under $\FPA$. 
Moreover, coequalisers in $\CSpa$ can be
constructed exactly as in $\BP$~\cite[Lemma
5.2]{CompletenessOfBPandCSpa}. Thus, $\CSpa$ is cocomplete
under $\NIDn{1}$ and hence under $\NIDf$ as well.
Therefore, the following are equivalent over $\ECST + \FPA$:
\begin{enumerate}
  \item $\CSpa$ has equalisers;
  \item $\CSpa$ is complete and cocomplete.
\end{enumerate}
\end{remark}

We summarise the equivalents of $\NIDf$.
\begin{theorem}
  \label{thm:NIDf}
The following are equivalent over $\ECST$:
\begin{enumerate}
\item $\NIDf$;
\item $\SGA$;
\item $\NIDn{n} \;(n \geq 2)$;
\item The class of models of any geometric theory is set-generated.
\end{enumerate}
Moreover, each of the following statement together with $\FPA$  is equivalent to $\NIDf$ over $\ECST$:
\begin{enumerate}
  \setcounter{enumi}{4}
  \item The class of formal points of any inductively generated formal
  topology is set-generated;
    \item The class of formal topology maps from a set-presented
      formal topology to an inductively generated formal
    topology is set-generated;
  \item $\CSpa$ has equalisers;
  \item $\CSpa$ is complete and cocomplete.
\end{enumerate}
\end{theorem}
\subsection*{Acknowledgements}
The authors were supported by the Japan Society for the Promotion of
Science (JSPS), Core-to-Core Program (A. Advanced Research
Networks).
The second author was supported by JSPS KAKENHI Grant Number JP16K05251.
The third author was supported by Istituto Nazionale di Alta
Matematica as a fellow of INdAM-COFUND-2012.


\begin{thebibliography}{10}

\bibitem{AspectofTopinCZF}
P.~Aczel.
\newblock Aspects of general topology in constructive set theory.
\newblock {\em Ann.\ Pure Appl.\ Logic}, 137(1-3):3--29, 2006.

\bibitem{SGA}
P.~Aczel, H.~Ishihara, T.~Nemoto, and Y.~Sangu.
\newblock Generalized geometric theories and set-generated classes.
\newblock {\em Math.\ Structures Comput.\ Sci.}, 25:1466--1483, 2015.

\bibitem{Aczel-Rathjen-Note}
P.~Aczel and M.~Rathjen.
\newblock Notes on constructive set theory.
\newblock Technical Report~40, Institut Mittag-Leffler, 2000/2001.

\bibitem{CTSbook}
P.~Aczel and M.~Rathjen.
\newblock {CST} {B}ook draft.
\newblock \url{http://www1.maths.leeds.ac.uk/˜rathjen/book.pdf}, August 2010.

\bibitem{CoquandIndGenFTop}
T.~Coquand, G.~Sambin, J.~Smith, and S.~Valentini.
\newblock Inductively generated formal topologies.
\newblock {\em Ann.\ Pure Appl.\ Logic}, 124(1-3):71--106, 2003.

\bibitem{GrandisFreydCompletion}
M.~Grandis.
\newblock Weak subobjects and the epi-monic completion of a category.
\newblock {\em J.\ Pure Appl.\ Algebra}, 154:193--212, 2000.

\bibitem{ConstRevMatheCompactness}
H.~Ishihara.
\newblock Constructive reverse mathematics: compactness properties.
\newblock In L.~Crosilla and P.~Schuster, editors, {\em From Sets and Types to
  Topology and Analysis: Towards Practicable Foundations for Constructive
  Mathematics}, number~48 in Oxford Logic Guides, pages 245--267. Oxford
  University Press, Oxford, 2005.

\bibitem{CompletenessOfBPandCSpa}
H.~Ishihara and T.~Kawai.
\newblock Completeness and cocompleteness of the categories of basic pairs and
  concrete spaces.
\newblock {\em Math.\ Structures Comput.\ Sci.}, 25:1626--1648, 2015.

\bibitem{IshiharaNemotoNIDFullness}
H.~Ishihara and T.~Nemoto.
\newblock Non-deterministic inductive definitions and fullness.
\newblock In {\em Concepts of proof in mathematics, philosophy, and computer
  science}, volume~6 of {\em Ontos Math. Log.}, pages 163--170. De Gruyter,
  Berlin, 2016.

\bibitem{IshiharaPalmgrenQuotient}
H.~Ishihara and E.~Palmgren.
\newblock Quotient topologies in constructive set theory and type theory.
\newblock {\em Ann.\ Pure Appl.\ Logic}, 141(1):257--265, 2006.

\bibitem{MartinLofMLTT}
P.~Martin-L{\"o}f.
\newblock {\em Intuitionistic type theory}.
\newblock Bibliopolis, Napoli, 1984.

\bibitem{PalmgrenCoequalisers}
E.~Palmgren.
\newblock Quotient spaces and coequalisers in formal topology.
\newblock {\em J.\ UCS}, 11(12):1996--2007, 2005.

\bibitem{PalmgrenMaxPartialPt}
E.~Palmgren.
\newblock Maximal and partial points in formal spaces.
\newblock {\em Ann.\ Pure Appl.\ Logic}, 137(1):291--298, 2006.

\bibitem{Sambin:intuitionistic_formal_space}
G.~Sambin.
\newblock Intuitionistic formal spaces --- a first communication.
\newblock In D.~G. Skordev, editor, {\em Mathematical Logic and its
  Applications}, volume 305, pages 187--204. Plenum Press, New York, 1987.

\bibitem{Sambin_BP_book}
G.~Sambin.
\newblock {\em Positive Topology and the Basic Picture. New structures emerging
  from constructive mathematics}.
\newblock Oxford University Press, to appear.

\bibitem{SimpsonSubsystem2ndOrderLogic}
S.~G. Simpson.
\newblock {\em Subsystems of Second Order Arithmetic}.
\newblock Perspectives in Logic. Cambridge University Press, Cambridge, second
  edition, 2009.

\bibitem{vandenBergNID}
B.~van~den Berg.
\newblock Non-deterministic inductive definitions.
\newblock {\em Arch.\ Math.\ Logic}, 52(1):113--135, 2013.

\bibitem{VeldmanFANKleene}
W.~Veldman.
\newblock Brouwer's {F}an {T}heorem as an axiom and as a contrast to {K}leene's
  alternative.
\newblock {\em Arch.\ Math.\ Logic}, 53(5):621--693, 2014.

\end{thebibliography}
\end{document}